\documentclass[12pt]{amsart}

\usepackage{parskip, amsthm,amsmath,amssymb,amscd,verbatim}
\usepackage[margin=0.7in]{geometry}
\usepackage{amsfonts, amssymb, amsmath, amsthm,  mathrsfs, stmaryrd, nopageno }
\usepackage[all]{xy}
\usepackage{verbatim}
\usepackage{relsize}
\usepackage{tikz}
\usepackage{hyperref}
\usepackage{graphicx}
\usepackage{hyperref}

\usepackage{enumitem,amssymb}
\newlist{todolist}{itemize}{2}
\setlist[todolist]{label=$\square$}
\usepackage{pifont}

\usetikzlibrary{matrix,arrows,decorations.pathmorphing}
  \theoremstyle{plain}
  \newtheorem{theorem}{Theorem}[section]
  \newtheorem{lemma}[theorem]{Lemma}

  \newtheorem{cor}[theorem]{Corollary}

  \theoremstyle{remark}

  \theoremstyle{definition}

  \numberwithin{equation}{section}
  \numberwithin{theorem}{section}

\newtheorem{proposition}[theorem]{Proposition}
\newtheorem{definition}[theorem]{Definition}

\newcommand{\Z}{\mathbb{Z}}

\newcommand{\Q}{\mathbb{Q}}

\newcommand{\C}{\mathbb{C}}

\newcommand{\xra}{\xrightarrow}

\DeclareMathOperator{\diag}{diag}

\DeclareMathOperator{\GL}{GL}
\DeclareMathOperator{\SL}{SL}

\DeclareMathOperator{\Ind}{Ind}

\DeclareMathOperator{\Aut}{Aut}
\DeclareMathOperator{\rdim}{rdim}

\input xy
\xyoption{all}

\begin{document}


\author{Jonathan Cohen}
\title{Minimal groups of given representation dimension}

\begin{abstract}
For a finite group $G$, let $\rdim(G)$ denote the smallest dimension of a faithful, complex linear representation of $G$. It is clear that $\rdim(H)\leq \rdim(G)$ for any subgroup $H$ of $G$. We consider $G$ with the property that $\rdim(H)<\rdim(G)$ whenever $H$ is a proper subgroup of $G$, in particular proving a classification of such groups when $G$ is abelian or $\rdim(G)\leq 3$. 
\end{abstract}

\maketitle

\section{Introduction}


We say a finite group $G$ has representation dimension $n$, written $\rdim(G)$, if there is a faithful representation $\rho: G \to \GL_n(\C)$ and no faithful representation of dimension $n-1$. Clearly $\rdim(H)\leq \rdim(G)$ for every subgroup $H$ of $G$. In this article we consider the class of groups $G$ such that $\rdim(H)< \rdim(G)$ for all proper subgroups $H$ of $G$; we call these groups {\it minimally faithful of degree} $\rdim(G)$. 

If $G$ is abelian or $\rdim(G)=2$, such groups are straightforward to classify; see Lemma \ref{abelian} and Proposition \ref{Degree 2}. The main result of this paper is a classification of minimally faithful groups of degree 3.  We summarize this result here; let $C_n$ be a cyclic group of order $n$ and $Q_8$ the quaternion group of order 8.



\begin{theorem}
Let $G$ be a minimally faithful group of degree 3. Then $G$ is one of the following:

a) $C_p\times C_p\times C_p$ for a prime $p$,

b) $(C_p\times C_p)\rtimes C_{3^k}$ for a prime $p\equiv -1\pmod 3$, $k\geq 1$, with center $C_{3^{k-1}}$,

c) $C_p\rtimes C_{3^k}$ for a prime $p\equiv 1\pmod 3$, $k\geq 1$, with center $C_{3^{k-1}}$,

d) $C_{3^k}\rtimes C_3$ for $k\geq 2$, with center $C_{3^{k-1}}$,

e) The Heisenberg group of order 27,

f) $Q_8\times C_p\times C_p$ for an odd prime $p$,

g) $(C_{2^k}\rtimes C_2)\times C_p\times C_p$ where $k\geq 2$, $p$ is an odd prime, and $C_2$ acts on $C_{2^k}$ by $g\mapsto g^{1+2^{k-1}}$,

h) $(C_q\rtimes C_{2^m})\times C_p\times C_p$ where $m\geq 1$, $q, p$ are distinct odd primes, and a generator of $C_{2^m}$ acts by inversion on $C_q$,

i) $(C_p\rtimes C_{2^m}) \times C_2$ where $m>1$, $p$ is an odd prime, and a generator of $C_{2^m}$ acts by inversion on $C_p$,   

j) $Q_8\times C_2$,

k) A nonabelian 2-group whose index-2 subgroups are all abelian with two invariant factors. 

\end{theorem}


\section{Preliminaries}

We write 
$D_{2n}$ for the dihedral group of order $2n$, 
$A_n$ for the alternating group of degree $n$, and $S_n$ for symmetric group of degree $n$. Let $G^1=[G, G]$ be the derived group of $G$, $G^2 =[G^1, G^1]$ for the derived group of $G^1$, and so on. The group $\GL_n(\C)$ is the group of $n\times n$ invertible matrices over the complex numbers, and its derived group $\SL_n(\C)$ is the subgroup of matrices with determinant 1. The group $PGL_n(\C)$ is the quotient of $\GL_n(\C)$ by its center. For a subset $S$ of a group $G$ we write $C_G(S)$ for the centralizer of $S$ in $G$ and $N_G(S)$ for its normalizer. We write $\Aut(G)$ for the automorphism group of a group $G$ and $Inn(G)\cong G/Z(G)$ for the group of inner automorphisms.

 
\begin{definition}
The group $G$ is {\it minimally faithful of degree} $n$ if $\rdim(G)=n$ and $\rdim(H)<n$ for all proper subgroups $H$ of $G$. 
\end{definition}

Our main results classify minimally faithful groups of degree 2 and 3. First we show that the case of abelian minimally faithful groups is straightforward. 

\begin{lemma} \label{abelian}
An abelian group $G$ is minimally faithful of degree $n$ if and only if it is an elementary abelian $p$-group of rank $n$, for some prime $p$. 
\end{lemma}

\begin{proof}
If $G = C_{n_1}\times \cdots \times C_{n_k}$ then $\rdim(G) \leq k$, with equality if and only if we have written $G$ with a minimal number of factors. Let the $p$-sylow subgroup of $G$ be written as $\prod\limits_{i=1}^{k_p} C_{p^{a_i}}$. The smallest number of factors in an expression for $G$ is $\max_p(k_p)$. If $p$ is such that $k_p = n$, then since $C_p^{n}\subset G$ and  $\rdim(C_p^n) = n= \rdim(G)$, the conclusion follows from minimal faithfulness. The converse is clear.  
\end{proof}


 We mention an easy and useful fact: if $G$ is a finite group and $n$ is coprime to $|Z(G)|$ then $\rdim(G\times C_n) = \rdim(G)$. This follows from the fact that any representation of $G$ can be extended to $G\times C_n$ by having $C_n$ act by scalars.

 Next, we recall a particular component of the classification of finite nonabelian groups containing only abelian proper subgroups.
 
 \begin{proposition}[\cite{MM03}]\label{AbelianSubgroupsReferenceProposition} If $G$ is a finite nonabelian group with only abelian proper subgroups then either $G$ is a $p$-group or else $G = Q\rtimes C_{p^a}$ where $Q$ is an elementary abelian $q$-group for a prime $q\neq p$, in which case we have $|Q|\equiv 1\pmod p$, and $Z(G)=C_{p^{a-1}}$. 
 \end{proposition}

\section{Minimally Faithful of Degree 2}

\begin{proposition} \label{Degree 2}
Suppose $G$ is a nonabelian group and is minimally faithful of degree $2$. Then either $G=Q_8$ or $G = \langle a, b| \ a^p = b^{2^m} = bab^{-1}a = 1 \rangle = C_p\rtimes C_{2^m}$ for some odd prime $p$ and some $m\geq 1$. 
\end{proposition}

\begin{proof}
A faithful representation $G\to \GL_2(\C)$ must be irreducible, since otherwise it is a sum of two 1-dimensional representations, which would imply that $G$ is abelian. Thus $G$ has even order. Let $P$ be a 2-Sylow subgroup of $G$. 

Suppose $G$ is not a 2-group. Then $\rdim(P)=1$ so $P$ is cyclic. By Cayley's normal 2-complement theorem, $G=K\rtimes P$ for some subgroup $K$ of odd order. Since $\rdim(K)=1$, $K$ is cyclic. Let $x\in P$ be a generator. If $x$ centralized $K$ then $G$ would be abelian, and if $x^2$ did not centralize $K$ then the proper subgroup generated by $K$ and $x^2$ would not be cyclic. Therefore $x$ acts on $K$ by an automorphism of order 2. We may decompose $K = K^+\times K^-$ where $xkx^{-1}= k^{\pm 1}$ for all $k\in K^{\pm}$. Since $G$ is nonabelian, $K^-$ is nontrivial. If $h\in K^-$ has prime order $p$ then $x$ and $h$ generate a nonabelian group, which must equal $G$. Therefore $G = C_p \rtimes C_{2^m}$ as claimed. These groups can be realized inside $\GL_2(\C)$: let $H$ be the cyclic group of order $2^{m-1}p$ generated by $x^2$ and $h$, take any faithful character $\chi:H\to \C^\times$, and then $\Ind_H^G(\chi)$ is a faithful 2-dimensional representation. 

Finally, suppose $G$ is a 2-group. Then $G$ has a cyclic maximal (index-2) subgroup. The $p$-groups with cyclic maximal subgroups have been classified; if $G\neq Q_8$ then one readily sees that $G$ contains $C_2\times C_2$, and if $G=Q_8$ then it is easy to check that $G$ is minimally faithful of degree 2. 
\end{proof}

\section{Degree 3}

We now begin the classification of minimally faithful groups of degree 3, first proving some easy elementary results. 



\begin{lemma}
Suppose $G$ is a nonabelian group. 

a) If $G$ has odd order or $Z(G)$ is not cyclic then $\rdim(G)\geq 3$. 

b) If $\rdim(G) = 3$ and $G$ has odd order then $3$ divides $|G|$ and $Z(G)$ is cyclic. 
\end{lemma}
\begin{proof}
a) Since $G$ is nonabelian, $\rdim(G)\geq 2$. If $G$ has odd order or $Z(G)$ is not cyclic then a faithful representation $\rho:G\to \GL_2(\C)$ would be reducible (the latter case because of Schur's lemma), hence would factor as a direct sum of two 1-dimensional representations, so $1\neq [G, G]\subset \ker(\rho)$, a contradiction.

b) If $\rdim(G)=3$ then a faithful representation $G\to \GL_3(\C)$ must be irreducible: it cannot have a 2-dimensional irreducible summand since $|G|$ is odd and cannot be a sum of 1-dimensional representations since $G$ is nonabelian. Thus $3$ divides $|G|$, and by Schur's lemma $Z(G)$ is cyclic. 
\end{proof}

\begin{cor}
Suppose $G$ is a nonabelian group and is minimally faithful of degree 3. Let $H$ be a proper subgroup of $G$.

a) If $H$ is nonabelian then $Z(H)$ is cyclic. 

b) If $H$ has odd order then $H$ is abelian. In particular, if $G$ has odd order then $H$ is abelian.

c) If $H$ is abelian then $H$ has at most 2 invariant factors. 

d) The order of $G$ is divisible by $2$ or $3$ (or both). 
\end{cor}

\begin{proof}
The first three parts are immediate since $\rdim(H)\leq 2$. For the last, observe that if $G$ has order coprime to 6 then any representation $G\to \GL_3(\C)$ must be a direct sum of three 1-dimensional representations. 
\end{proof}



\begin{lemma} \label{evenorderreducibility}
Suppose that $G$ is minimally faithful of degree 3. If $Z(G)\cap [G, G]$ has even order then any faithful representation $\rho:G\to \GL_3(\C)$ is reducible and $[G, G]$ is isomorphic to a subgroup of $\SL_2(\C)$.
\end{lemma}

\begin{proof}
Suppose $z\in Z(G)\cap [G, G]$ has order 2. If $\rho$ is irreducible then $\rho(z)=-I_3$ by Schur's lemma. On the other hand, $\rho([G, G])\subset \SL_3(\C)$ and $-I_3\not\in \SL_3(\C)$. So $\rho$ is reducible, $G\cong \rho(G)\subset \GL_2(\C)\times \GL_1(\C)$, and $\rho([G, G])$ is contained in the derived group of $\GL_2(\C)\times \GL_1(\C)$, which is $\SL_2(\C)\times \{1\}$. 
\end{proof}

\begin{theorem}
Let $G$ be minimally faithful of degree 3. Then $[G, G]$ is abelian. 
\end{theorem}

\begin{proof}
We may assume $G$ is nonabelian. Let $G^1:=[G, G]$. 

{\bf STEP 1}: $G\neq G^1$.  

Suppose $G = G^1$. Then any faithful representation $\rho:G\to \GL_3(\C)$ has image in $\SL_3(\C)$. We refer to the classification of finite subgroups of $\SL_3(\C)$, up to $\GL_3(\C)$-conjugacy, found in \cite{SL3C}. If $\rho$ is reducible then (up to conjugacy) $\rho(G)$ is contained in the image of $g\mapsto (g, \det(g)^{-1}):\GL_2(\C)\to (\GL_2(\C)\times \GL_1(\C) )\cap \SL_3(\C)$, which would imply $\rdim(G)=2$. Therefore $\rho$ is irreducible. Among the subgroups of $\SL_3(\C)$ that act irreducibly and primitively on $\C^3$, the groups $A_5$, $PSL_2(7)$, and $3A_6$ (the nontrivial triple cover of $A_6$) all contain $A_4$, while the others all have a nonabelian 3-Sylow subgroup (and are not themselves 3-groups). Since $\rdim(A_4)=3$, none of these groups are minimally faithful of degree 3. The same issue occurs if $G = H \cdot Z(\SL_3(\C))$ where $H = A_5$ or $PSL_2(7)$. Therefore $G$ acts irreducibly and imprimitively on $\C^3$. 

Thus $\rho(G)$ is generated by a finite subgroup $D$ of the diagonal matrices and a cyclic permutation $p=\begin{bmatrix}0 & 1 & 0\\ 0 & 0 & 1 \\1 & 0 & 0 \end{bmatrix}$, or else is generated by such a group together with a matrix $Q=\begin{bmatrix}0 & a & 0\\ b & 0 & 0 \\0 & 0 & c \end{bmatrix}$ with $abc=-1$. In the both cases $D$ is a normal subgroup of $\rho(G)$. In the first case $\rho(G)/D= C_3$, so $G\neq G^1$. In the second case, since $Q^2\in D$ and $QpQ^{-1}\in Dp^{-1}$, we have $\rho(G)/D = S_3$, which admits a nontrivial character, so $G\neq G^1$.

{\bf STEP 2}: $G$ is solvable.

If $G$ is not solvable then $G^1$ is not solvable. Since $G^1\neq G$ we can regard $G^1$ as a subgroup of $\GL_2(\C)$. Then $G^2 := [G^1, G^1]\subset \SL_2(\C)$ is not solvable. The unique nonsolvable subgroup of $\SL_2(\C)$ is the binary icosahedral group, which we will denote by $BI_{120}$. Thus $Z(G^2) \cong Z(BI_{120})$ has order 2 which implies $Z(G^2)\subset Z(G)\cap [G, G]$. By lemma \ref{evenorderreducibility}, $\rho(G)$ may be regarded as a subgroup of $\GL_2(\C)\times \GL_1(\C)$ and $\rho(G^1)\subset \SL_2(\C)\times \{1\}$ so $G^1 \cong BI_{120}$. If $(M, t)\in \rho(G\setminus G^1)$ where $M\in \GL_2(\C)$, $t\in \C^\times$, and $(g, 1)\in \rho(G^1)$, then $(MgM^{-1}, 1)\in \rho(G^1)$. so $(M, 1)$ and $\rho(G^1)$ generate a finite subgroup of $\GL_2(\C)\times\{1\}$ which contains a copy of $BI_{120}$. The only such groups are generated by $BI_{120}$ and a subgroup of scalar matrices, by \cite{GL2C}. 

Therefore, since $(M, t)\not\in \rho(G^1)$, $(M, t)$ commutes with $\rho(G^1)$, which implies $M$ is a scalar matrix by Schur's lemma. So $G = Z(G)G^1$ with $Z(G)\cap G^1$ of order 2. If $Z(G)$ is cyclic then $\rdim(G)=2$ since we can let $Z(G)$ act by scalars on $\C^2$ in a realization of $BI_{120}\subset \SL_2(\C)$. Suppose $Z(G)$ is not cyclic. Since $Q_8$ is the 2-Sylow subgroup of $BI_{120}$, we have the nonabelian subgroup $Q_8Z(G)$ with noncyclic center. Thus $\rdim(Q_8Z(G))=3$. But $Q_8Z(G)\neq G$ since $Q_8$ is not centralized by $BI_{120}$, contradicting minimality of $G$.

{\bf STEP 3}: If $G^1$ is not abelian then $G^1 = Q_8$ or $G^1 = Q_8\rtimes C_3$.

Suppose $G^1$ is nonabelian. Note $G^2=[G^1, G^1] \subset \SL_2(\C)$. The finite subgroups of $\SL_2(\C)$ are $C_n$, $BD_{4n} := \langle A, B: A^n = B^2, \ B^4 = 1 \ BAB^{-1} = A^{-1} \rangle$ (dicyclic or binary dihedral groups of order $4n$, $n\geq 2$), $Q_8\rtimes C_3$ (the binary tetrahedral group), $BO_{48}$ (the binary octahedral group), and $BI_{120}$ (the binary icosahedral group). Their derived groups are, respectively, $\{1\}$, $C_{2n}$ (respectively, $C_{n}$) if $n$ is odd (respectively, even), $Q_8$, $Q_8\rtimes C_3$, and $BI_{120}$. The finite subgroups of $PGL_2(\C)$ are $C_n$, $D_{2n}$, $A_4$, $S_4$, and $A_5$. In particular, $G^1/Z(G^1)$ and $G^2/(Z(G^2)\cap \{\pm I_2\}) = [G^1/Z(G^1), G^1/Z(G^1)]$ are such subgroups of $PGL_2(\C)$. If $Z(G^2)$ has a unique element $z$ of order 2 then $z\in Z(G)\cap [G, G]$ and lemma \ref{evenorderreducibility} applies. This holds except when $G^2$ is cyclic of odd order. 


Case 1: Since $G$ is solvable, $G^2\neq BI_{120}$. 

Case 2: If $G^2 = BO_{48}$ then $G^2/Z(G^2)\cong S_4$, so $G^1/Z(G^1)\subset PGL_2(\C)$ must contain $S_4$, hence equal $S_4$. Thus $G^2/Z(G^2)\subset [S_4, S_4] = A_4$, a contradiction.  

Case 3: If $G^2 = Q_8\rtimes C_3$ then $G^1\subset \SL_2(\C)$ equals $BO_{48}$, the only subgroup of $\SL_2(\C)$ with the correct commutator subgroup. We claim that $BO_{48}$ cannot be realized as a commutator subgroup of any group, hence cannot equal $G^1$. We have $Inn(BO_{48}) =BO_{48}/(\pm I_2) = S_4$ and $\Aut(BO_{48}) = S_4\times C_2$.  
But then we cannot have $Inn(BO_{48})\subset \Aut(BO_{48})^1 = A_4$, which needs to occur since the image of the $G^1$ under the natural map $G\to \Aut(G^1)$ equals $Inn(G^1)$ and lies in $\Aut(G^1)^1$.

Case 4: If $G^2 = BD_{4n}$ with $n>2$ then $G^1\subset \SL_2(\C)$ cannot exist because $BD_{4n}$ is not the commutator subgroup of any finite subgroup of $\SL_2(\C)$. If $n=2$ then $G^2 = Q_8$ and $G^1 = Q_8\rtimes C_3$, since the binary tetrahedral group is the only subgroup of $\SL_2(\C)$ with the correct commutator subgroup.   

Case 5i: If $G^2=C_n$ is cyclic of even order $n>1$ then $G^1\subset \SL_2(\C)$ and has derived group $C_n$. Therefore $G^1 = BD_{4n}$. Realizing $G$ as a subgroup of $\GL_2(\C)\times \GL_1(\C)$ (using lemma \ref{evenorderreducibility}), let $H$ be the subgroup of $\GL_2(\C)$ generated by the first components of elements of $G$. Clearly $H^1 = G^1$. If $n\geq 4$ then $H/Z(H)\subset PGL_2(\C)$ cannot have the dihedral group $D_{2n}$ as its derived group. Since $G^1/Z(G^1)=D_{2n}$ this is a contradiction. If $n=2$, then $G^1 = Q_8$ since this is the unique subgroup of $\SL_2(\C)$ with commutator subgroup of order 2.

Case 5ii: Now assume $G^2=C_n$ is cyclic of odd order $n>1$. Since $G^1$ is a nonabelian subgroup of $G$, $Z(G^1)$ is cyclic. Suppose $Z(G^1)$ has even order. Then its unique element of order 2 will lie in $Z(G)$, so lemma \ref{evenorderreducibility} applies and $G^1\subset \SL_2(\C)$. The only subgroups of $\SL_2(\C)$ with cyclic derived group are $BD_{4k}$, but $BD_{4k}^1 = C_{2k}$ if $k$ is odd and $BD_{4k} = C_k$ if $k$ is even, so the odd-order cyclic group $G^2$ is not a derived group of $G^1\subset \SL_2(\C)$, which is absurd. Thus $Z(G^1)$ has odd order. Now $G^1\subset \GL_2(\C)$ is nonabelian so $G^1/Z(G^1)\subset PGL_2(\C)$, and $(G^1/Z(G^1))^1 = G^2$ so $G^1/Z(G^1) = D_{2n}$, the dihedral group of order $2n$. From the classification of finite subgroups of $\GL_2(\C)$ in \cite{GL2C},  
$G^1= C_r \times D_{2n} = C_r\times (G^2\rtimes C_2)$ where $r$ is odd (all the other subgroups contain $-I_2$). But then $G^1/G^2=C_{2r}$ and $G^2=C_n$ are both cyclic, which implies that $G^1$ is cyclic and $G^2$ is trivial [proof: $G/C_G(G^2)$ is a subgroup of $\Aut(G^2)$, which is abelian, hence $G^1\subset C_G(G^2) $, so $G^2\subset Z(G^1)$. But then $G^1/G^2$ being cyclic implies $G^1/Z(G^1)$ is cyclic, so $G^1$ is cyclic]. 


{\bf STEP 4}: $G^1\neq Q_8$ and $G^1 \neq Q_8\rtimes C_3$. 

Suppose $G^1$ is either of these two groups. We have $\Aut(Q_8)\cong S_4\cong \Aut(Q_8\rtimes C_3)$. Therefore if $P$ is a $p$-Sylow subgroup of $G$ with $p\geq 5$ then $P\subset C_G(G^1)$. Since $G^1$ is nonabelian, and $P\subset Z(PG^1)$, ($P$ is abelian since it is a proper subgroup of odd order) $P$ is cyclic. Since $\rho(G^1)\subset \GL_2(\C)\times \{1\}$ is nonabelian and $\rho(G)\subset \GL_2(\C)\times \GL_1(\C)$, Schur's lemma implies that $\rho(C_G(G^1))$ is a group of matrices of the form $\diag(x, x, y)$, so $P\subset Z(G)$. Because $P\cap G^1 = 1$, we have $G/G^1\cong F\times P$ for an abelian group $F$. If $N$ is the preimage of $F$ in $G$, then $G = N\times P$. Clearly $|P|$ and $|Z(N)|$ are coprime, so $\rdim(N\times P) = \rdim(P)$. Therefore $P$ is trivial and the only primes that may divide $|G|$ are 2 and 3. 

Since $G^1$ contains a unique element $z$ of order 2, we have $z\in Z(G)$ and Schur's lemma forces $\rho(z) = \diag(-1, -1, 1)$. Let $K = \ker(\rho(G)\to \GL_2(\C))$, where the map is the projection onto the first factor. The group $K$ is nontrivial since $\rdim(G)>2$ and $K\subset \rho(Z(G))$ since $K$ consists of elements of the form $\diag(1, 1, y)$. If $k\in K$ had order 2 then $k$ and $\rho(z)$ would generate a noncyclic subgroup of $\rho(Z(G))$. Thus $G^1Z(G)$ would be nonabelian with noncyclic center, hence would equal $G$. But then $G^1 = (G^1Z(G))^1 \subset G^2Z(G)$, which is false for both $G^1 = Q_8$ and $G^1 = Q_8\rtimes C_3$. Therefore $K$ has odd order, hence is a 3-group, and if $k\in K$ has order 3 then $k=\diag(1,1, \omega)$ for a primitive 3rd root of unity $\omega$. 

Suppose there exists a scalar matrix $xI_3\in \rho(Z(G))$ for some $1\neq x\in \C^\times$. Replacing $x$ by a power if necessary, we can assume that $x=-1$ or $x=\omega$. In any case, $xI_3$, $\rho(z)$, and $\diag(1,1, \omega)$ will generate a noncyclic subgroup of $\rho(Z(G))$, yielding the same contradiction as above. Therefore $\rho(G)$ contains no nontrivial scalar matrices. If we write $\rho = \sigma \oplus \chi:G\to \GL_2(\C)\times \GL_1(\C)$, then the representation $\sigma\otimes \chi^{-1}:G\to \GL_2(\C)$ has trivial kernel, since $\sigma(g)\chi^{-1}(g) = I_2$ if and only if $\sigma(g) = \chi(g)I_2$ if and only if $\rho(g) = \chi(g)I_3$ is a scalar matrix. Thus $\rdim(G)<3$, a contradiction. 
\end{proof}

Now we will break into cases depending on the rank of the 2-Sylow subgroup of the abelian group $G^1=[G, G]$. First we prove two auxiliary results.

\begin{lemma} \label{Lemmacyclic2Sylow}
Suppose $G$ is minimally faithful of degree 3, with even order and cyclic 2-Sylow subgroups. Then 
$G=(C_q\rtimes C_{2^m})\times C_p\times C_p$
where $q, p$ are distinct odd primes, $m\geq 1$, and a generator of $C_{2^m}$ acts by inversion on $C_q$. 
\end{lemma}

\begin{proof}
Let $P$ be a 2-Sylow subgroup of $G$ and assume $P$ is cyclic. By Cayley's normal 2-complement theorem, $G \cong K\rtimes P$, where $K$ is a normal 2-complement. Since $K\neq G$ and has odd order, $K$ is abelian. Let $\phi: P \to \Aut(K)$ be the homomorphism defining $G$, let $x$ be a generator of $P$ and $n$ be the order of $\phi(x)$. Finally, let $H$ be the subgroup generated by $x^2$ and $K$, which is the unique subgroup of index 2 in $G$.

If $n=2$ then we claim $\rdim(G)\neq 3$ or $Z(G)$ is not cyclic. In this case $H$ is abelian. If $K$ is cyclic then so is $H$, and there is a faithful character $\chi:H\to \C^\times$. But then $\Ind_H^G(\chi)$ is a faithful 2-dimensional representation of $G$. Thus $K$ is not cyclic. Since $K$ has odd order and $x$ acts on $K$ by an automorphism of order 2, we can decompose $K = K^+\times K^-$ where $K^{\pm} := \{ k\in K: xkx^{-1} = k^{\pm 1}\}$. Since $K^+$ is centralized by $H$ and $x$, it lies in $Z(G)$. Clearly $K^-$ is normalized by $P$, so 
$$G \cong K^+\times (K^-\rtimes P).$$ 
Suppose $K^-$ is not cyclic, so it has two invariant factors. For independent generators $a$ and $b$ of $K^-$ we have $xax^{-1} = a^{- 1}$ and $xbx^{-1} = b^{- 1}$. Suppose $\rho:G\to \GL_3(\C)$ was faithful. Since $a$ and $a^{-1}$ are conjugate in $G$, the set of eigenvalues of $\rho(a)$ is equal to the set of its reciprocals. Since $a$ has odd order, a nontrivial eigenvalue of $\rho(a)$ cannot be equal to its inverse. Thus the eigenvalues of $\rho(a)$ are $\zeta_a^{\pm 1}, 1$, for some root of unity $\zeta_a$ of order equal to the order of $a$ in $G$. Similar considerations hold for $\rho(b)$. Since $K^-$ is not cyclic there is a common prime dividing the orders of $a$ and of $b$. This, the commutativity of $K$, and the assumption that $\rho$ is faithful means we can assume without loss of generality that $\rho(a) = \diag(\zeta_a, \zeta_a^{-1}, 1)$ and $\rho(b) = \diag(\zeta_b, 1, \zeta_b^{-1})$ (we cannot put the 1 in the same entry for both since then $\rho$ would not be faithful). But there is no matrix in $\GL_3(\C)$ that simultaneously conjugates both of the matrices $\rho(a)$ and $\rho(b)$ to their inverses: the centralizers of either matrix consist solely of diagonal matrices, so any matrix that conjugates $\rho(b)$ to $\rho(b)^{-1}$ is a diagonal matrix times $\begin{bmatrix} & & 1 \\ & 1 & \\ 1 & &  \end{bmatrix}$, but no such matrix can conjugate $\rho(a)$ to its inverse. Therefore $\rho(x)$ cannot be defined, so either $K^-$ is not cyclic or $\rdim(G)>3$. So $K^-$ is cyclic. Observe that $\rdim(K^-\rtimes P)=2$: let $N$ be the cyclic group $\langle K^-, x^2 \rangle$ and $\chi:N\to \C^\times$ a faithful character. Then $\Ind_N^{K^-\rtimes P}(\chi)$ is a faithful 2-dimensional representation.  Therefore $K^+$ is nontrivial. If $K^+$ is cyclic then $\rdim(G) = \rdim(K^-\rtimes P)=2$. Thus $K^+$, and hence $Z(G)$, is not cyclic. There exists an odd prime $p$ such that $C_p\times C_p\subset K^+$; the group $(C_p\times C_p) \times (K^- \rtimes P)$ is nonabelian with noncyclic center, so has representation dimension at least 3, so equals $G$ by minimality. Noting that every subgroup of $K^-$ is normalized by $P$ by definition, if $q$ is an odd prime such that $C_q\subset K^-$, then $(C_p\times C_p) \times (C_q \rtimes P)$ is nonabelian with noncyclic center, hence equals $G$.

If $n\geq 4$ then we claim $\rdim(G)>3$. First assume $n=4$. Suppose there is a cyclic subgroup $L\subset K$ which is normalized by $x$ on which conjugation by $x$ is an automorphism of order 4. We claim that $\rdim(PL)>3$ (in fact it equals 4). If $y$ generates the cyclic group $L$ then $xyx^{-1} = y^t$ where $t^2 \equiv -1 \pmod p$ for some prime $p$ dividing $|L|$. Thus $y$ is conjugate its inverse by $x^2$. If $\rho:PL\to \GL_3(\C)$ is a faithful representation then, as above, the eigenvalues of $\rho(y)$ are $\zeta^{\pm 1}, 1$ for some root of unity $\zeta$ of order equal to $|L|$. But since $y$ is conjugate to $y^t$, $\rho(y^t)$ has the same eigenvalues as $\rho(y)$, which forces $\zeta^t = \zeta^{-1}$ or $\zeta^t = \zeta$ or $\zeta^t=1$. These are all absurd, 
 so $L$ does not exist. Since $x^2$ acts on $K$ by an automorphism of order 2, we decompose $K=K^+\times K^-$ where $K^{\pm}:= \{k\in K: x^2kx^{-2} = k^{\pm 1} \}$. If $K^-$ is not cyclic then the proof of the $n=2$ case shows that the group generated by $x^2$ and $K^-$ has representation dimension greater than 3. If $K^-$ is trivial then $x$ acts on $K$ by an element of order 2, contradicting the assumption that $n=4$. So $K^-$ is cyclic and nontrivial. Note that $K^-$ is normalized by $x$ since $x^2(xkx^{-1})x^{-2} = xk^{-1}x^{-1}= (xkx^{-1})^{-1}$ for $k\in K^-$. But now $K^-$ satisfies the conditions of the group $L$ above, so cannot exist. Finally, assume $n\geq 8$. Then $x^2$ acts on $K$ by an automorphism of order 4, so the previous case shows that $\rdim(H)>3$, hence $\rdim(G)>3$ (note that the proof of the $n=4$ case made no use of the minimality of $G$ to derive $\rdim(G)>3$). 
\end{proof}

\begin{lemma} \label{2-group rdim2 lemma}
Suppose $G$ is a nonabelian 2-group containing only abelian proper subgroups and $\rdim(G)=2$. Then either $G=Q_8$ or $G=C_{2^k}\rtimes C_2$ where $k\geq 2$ and $C_2$ acts on $C_{2^k}$ by $g\mapsto g^{1+2^{k-1}}$. 

\end{lemma}

\begin{proof}
If $G$ only has cyclic proper subgroups then $G$ is minimally faithful of degree 2, so by our earlier classification $G=Q_8$. If $G\supset V= C_2\times C_2$ then we can choose an index-2 subgroup $H\subset G$ with $H\supset V$. The group $H$ is abelian by assumption, and $V$ is the characteristic subgroup of $H$ generated by elements of order 2, so $N_G(V)=G$. The map $G\to \Aut(V)\cong S_3$ factors through $G/H$ and is nontrivial because $Z(G)$ is cyclic. Thus there exists a unique $1\neq z\in V\cap Z(G)$. Let $v$ and $vz$ be the other nontrivial elements of $V$, and let $g\in G\setminus H$. We have $vgvg^{-1}  = z$. If $g^2 = 1$ then $g$ and $V$ generate the nonabelian group $C_4\rtimes C_2$, which must therefore equal $G$. If $g^2\neq 1$. then some power of $g^{2^{k-1}}.$ is nontrivial and lies in $V$, hence is equal to $z$ since $g$ does not commute with $v$ or $vz$. This yields the groups $\langle g, v: g^{2^k} = v^2 = 1, vgv = g^{1 + 2^{k-1}} \rangle $, where $k\geq 2$. Notice these groups have representation dimension $2$ since we can induce a faithful character of the index-2 cyclic 
subgroup generated by $g$ up to $G$. 
\end{proof}

\begin{theorem} \label{2-rank trichotomy lemma}
Let $G$ be nonabelian and minimally faithful of degree 3, and let $0\leq  r\leq 2$ be the rank of the $2$-Sylow subgroup of $G^1$. 

A) If $r=2$
then $G = (C_2\times C_2)\rtimes C_{3^k}$ for some $k\geq 1$. 

B) If $r=1$ 
then either 

	i) $G$ is a 2-group, 

	ii) $G = Q_8\times C_p\times C_p$ for an odd prime $p$, or

	iii) $G  = (C_{2^k}\rtimes C_2)\times C_p\times C_p$ for an odd prime $p$, where $k\geq 2$ and $C_2$ acts on $C_{2^k}$ by $g\mapsto g^{1+2^{k-1}}$.  

C) If $r=0$ then either 

i) $G$ has odd order, 

ii) $G$ is the group appearing in lemma \ref{Lemmacyclic2Sylow}, or 

iii) $G= (C_p\rtimes C_{2^m}) \times C_2$ where $m>1$, $p$ is an odd prime, and a generator of $C_{2^m}$ acts by inversion on $C_p$. 

\end{theorem}

\begin{proof}
Let $\rho:G\to \GL_3(\C)$ be a faithful representation of $G$. 

A) The elements of order exactly 2 inside $G^1$ generate a subgroup $V\cong C_2\times C_2$; it is characteristic in $G^1$ so $N_G(V) =G$. Since $\Aut(V)=S_3$, the index $[G: C_G(V)]$ divides 6. The centralizer $C_G(V)$ is abelian since $V\subset Z(C_G(V))$ and $V$ is not cyclic. If $\rho$ was reducible then $G^1$ would be an abelian subgroup of $\SL_2(\C)$, hence cyclic. But $G^1\supset V$, so $\rho$ is irreducible. Thus $3$ divides $|G|$ and $G$ contains a nontrivial 3-Sylow subgroup $P$. If $P\subset C_G(V)$ then $[G:C_G(V)]=2$, since the image of the map $G\to \Aut(V)$ could contain no element of order 3. But if $G$ has an index-2 abelian subgroup then every irreducible representation of $G$ has degree dividing 2, contradicting the irreducibility of $\rho$. Thus $P$ does not centralize $V$. Let $x\in P\setminus C_G(V)$. Then $x$ and $V$ generate a group isomorphic to $R:=(C_2\times C_2)\rtimes C_{3^k}$. Every proper subgroup of $R$ is abelian of rank at most 2, so it suffices to show that $\rdim(R)=3$. 

The group $R$ has a faithful 3-dimensional representation: let $\chi$ be a character of the index-3 abelian subgroup $H:=C_2\times C_2\times C_{3^{k-1}}$ with $\ker(\chi)$ of order 2, then $\rho=Ind_H^R(\chi)$ is readily checked to be faithful. We claim that the group $R$ has no faithful 2-dimensional representation. The subgroup $C_2\times C_2$ is abelian so we can assume it acts diagonally on $\C^2$, but the only diagonal $2\times 2$ matrices of order 2 are $\diag(1, -1), \diag(-1, 1), \diag(-1, -1)$. The third matrix must lie in the center but no order-2 element of $R$ is central. So $\rdim(R)=3$ and $R=G$.

B) We will first show that either $G$ is a 2-group or $G = P \times C_p\times C_p$ for a 2-Sylow group $P$ and odd prime $p$. Since $r=1$, $G^1$ has a nontrivial cyclic 2-Sylow subgroup. Hence $G^1$ contains a unique element $z$ of order 2, lemma \ref{evenorderreducibility} applies: $\rho$ is reducible and $\rho(G^1)\subset \SL_2(\C)\times \{1\}$, so $G^1$ is cyclic. Therefore $\rho(z) = \diag(-1, -1, 1)$, where without loss of generality we have assumed $\rho(G)\subset \GL_2(\C)\times \GL_1(\C)$. 

Suppose that $|G^1|=2$. Then $G^1\subset Z(G)$, and $G=P \times K$ where $K$ is abelian of odd order [proof: let $a\in G$ and let $b\in G$ of odd order. If $aba^{-1} = bz$, then the left side has odd order and the right side even order. Thus $aba^{-1}b^{-1}=1$]. If $K\supset C_p\times C_p$ for an odd prime $p$ then $P\times C_p\times C_p$ is nonabelian with noncyclic center, so has representation dimension 3 and must equal $G$. Otherwise $K$ is cyclic, which implies $\rdim(P\times K) = \rdim(P)$, so by minimality $G=P$. 

Now suppose $g$ generates the cyclic group $G^1$ and has (even) order greater than two. Since $G^1$ is abelian we can assume $\rho(G^1)$ consists of diagonal matrices. Thus $\rho(g) = \diag(x, x^{-1}, 1)$ for some $x\neq \pm 1$ and $\rho(C_G(G^1))$ is contained in the abelian group of diagonal matrices. If $n\in G\setminus C_G(G^1)$, then $\rho(ngn^{-1}) = \diag(y, y^{-1}, 1)$, since it must lie in $\rho(G^1)$. The eigenvalues of $\rho(g)$ and $\rho(ngn^{-1})$ are the same but $g\neq ngn^{-1}$ so $y=x^{-1}$. Thus $n$ acts on $G^1$ by inversion. Since $n$ was an arbitrary element of $G\setminus C_G(G^1)$, we have $[G:C_G(G^1)]=2$.

Write $C_G(G^1)  = Q\times K =  Q\times K^+\times K^-$ where $Q$ is the 2-sylow subgroup of $C_G(G^1)$, $K$ is its subgroup of odd-order elements, and $K^\pm := \{ k\in K: nkn^{-1} = k^{\pm 1}  \}$, noting that the action of $n$ on $K$ is by an element of order 2 since $[G:C_G(G^1)]=2$. Both $Q$ and $K^\pm$ are normal in $G$, with $Q$ being characteristic in the abelian group $C_G(G^1)$ and $K^\pm$ being normalized by definition. Replacing $n$ by an odd power if necessary, we can assume that the order of $n$ is a power of $2$. Let $P= \langle Q, n\rangle$, so $P$ is a 2-Sylow subgroup of $G$ and 
$$G= K^+\times ( K^-\rtimes P  ).$$ 
If $k\in K^-$ then $nkn^{-1}k^{-1} = k^{-2}\in G^1$, but $k$ has odd order so $K^- \subset G^1$. Every nontrivial commutator has the form $ana^{-1}n^{-1}$ for some $a\in Q\times K^-$. If $Q\subset Z(P)$ then the nontrivial commutators would all have odd order, contradicting our initial assumption. Thus $P$ is nonabelian. 

If $K^+$ is not cyclic then it contains $C_p\times C_p$ for an odd prime $p$, so $C_p\times C_p\times P$ is a nonabelian subgroup of $G$ with noncyclic center, so has representation dimension 3 and must equal $G$. If $K^+$ is cyclic and nontrivial then $\rdim(G) = \rdim ( K^-\rtimes P  )$ since $K^-\cap Z(G)=\{1\}$, but $\rdim ( K^-\rtimes P  )<\rdim(G)$ by minimality so this is absurd. 

Now assume that $K^+$ is trivial. We will show that $Z(P)$ is not cyclic, which forces $G=P$. Let $L$ be the kernel of the projection $\rho(G)\to \GL_2(\C)$. So $L\subset \rho(Z(G))$ is those matrices of the form $\diag(1, 1, y)$. If $L$ was trivial then $\rdim(G)<3$. Since $\rho(K^-)\subset \rho(G^1)\subset \SL_2(\C)\times\{1\}$, we have $\rho(K^-)\cap \rho(L)= \{1\}$. Hence $|L|$ is a nontrivial power of 2, so $L\ni \diag(1, 1, -1)$, which together with $\rho(z)= \diag(-1, -1, 1)$ creates a noncyclic subgroup of $\rho(Z(G)\cap P)\subset \rho(Z(P))$. 

We have shown that if $G$ is not a 2-group then $G = Q\times C_p\times C_p$ for a 2-group $Q$ and odd prime $p$. Suppose $G$ is not a 2-group. Clearly the 2-Sylow subgroup $Q$ of $G$ must be nonabelian. If $Q$ contained a nonabelian proper subgroup $Q_1$, then $Q_1\times C_p\times C_p$ would be nonabelian with noncyclic center, hence would have representation dimension 3, contradicting minimality of $G$. Thus every proper subgroup of $Q$ is abelian. Since $\rdim(Q)=2$, lemma \ref{2-group rdim2 lemma} gives the two possible structures of $Q$, hence of $G$. That these groups both have a faithful representation of dimension 3 is simple: let $\rho = \sigma\oplus \chi: G\to \GL_3(\C)$ where $\sigma: Q \times C_p \to \GL_2(\C)$ and $\chi: C_p\to \GL_1(\C)$ are faithful.

C) Let $P$ be a 2-Sylow subgroup of $G$, so $P$ is abelian since $G^1$ has odd order. 
If $P$ is trivial then $|G|$ is odd. If $P$ is cyclic then lemma \ref{Lemmacyclic2Sylow} applies. So we may assume $P$ has two invariant factors. If $g\in G$ has odd order then $g$ and $G^1$ generate an odd-order subgroup of $G$. Since $P$ is nontrivial, any odd-order subgroup of $G$ is proper, hence abelian. Thus $g\in C_G(G^1)$ and $C_G(G^1)$ contains all elements of $G$ of odd order. Therefore $G = PC_G(G^1)$. We have $P\supset V\cong C_2\times C_2$ since $P$ is not cyclic. 

Suppose that there exists $ v\in V\setminus Z(G)$. From this we will derive a contradiction. Since $G^1$ is abelian of odd order, $G^1 = K^+\times K^-$ where $K^\pm := \{k\in G^1: vkv = k^{\pm 1} \}$. Note $K^-$ is nontrivial since $v\not\in Z(G)$. If $1\neq k\in K^-$ then $\rho(k)\in \rho(G^1)\subset \SL_3(\C)$ is conjugate to its inverse, so its eigenvalues are $\zeta^{\pm1}$, 1 for some odd root of unity $\zeta\neq 1$. Choosing a basis so that $\rho(k) = \diag(\zeta, \zeta^{-1}, 1)$, we see $\rho(C_G(k))$ is contained in the abelian subgroup of diagonal matrices. Thus $C_G(G^1) = C_G(k)$ (both consist under $\rho$ of the diagonal matrices in $\rho(G)$) is an abelian normal subgroup of $G$. A conjugate of $\rho(k)$ under $\rho(G)$ lies in $\rho(G^1)$ so it is diagonal with the same eigenvalues. The permutations of the diagonal entries of $\rho(k)$  induces a homomorphism $G\to S_3$ with kernel $C_G(G^1)$. So $[G:C_G(G^1)]$ divides 6 and is divisible by 2. Since $[G:C_G(G^1)] = [P:P\cap C_G(G^1)]$ is a power of 2, we have $[G:C_G(G^1)]=2$. Now $v\not\in C_G(G^1)$ and $v$ has order 2 so $G = C_G(G^1)\rtimes \langle v\rangle $. Let $x\in C_G(G^1)\cap P$ be the generator of the nontrivial and cyclic 2-Sylow subgroup of $C_G(G^1)$. 
If $T$ is the subgroup of $C_G(G^1)$ consisting of odd-order elements then $G = \langle x \rangle \times (T\rtimes \langle v\rangle )$. Since $Z(T\rtimes \langle v\rangle) $ has odd order, $\rdim (G) = \rdim (T\rtimes \langle v\rangle)$. But $T\rtimes \langle v\rangle$ is a proper subgroup of $G$, contradicting minimality of $G$. 

Thus $V\subset Z(G)$. In particular, $Z(G)$ is not cyclic so by Schur's lemma $\rho$ is reducible. We can assume without loss of generality that $\rho(G)\subset \GL_2(\C)\times \GL_1(\C)$, hence $\rho(G^1)\subset \SL_2(\C)\times \{1\}$. In particular, $G^1$ is cyclic. Let $g\in G^1$ be nontrivial, so that we may assume $\rho(g) = \diag(\zeta, \zeta^{-1}, 1)$ for some odd root of unity $\zeta\neq 1$. It follows that $\rho(C_G(G^1))$ consists of the diagonal matrices in $\rho(G)$, hence $C_G(G^1)$ is abelian. Since $G = PC_G(G^1)$ is nonabelian and $P$ is abelian, there exists $y\in P\setminus C_G(G^1)$. Since $\rho(ygy^{-1})$ has the same eigenvalues as $\rho(g)$, is not equal to $\rho(g)$, and lies in $\rho(G^1)$, we must have $ygy^{-1}=g^{-1}$. The subgroup of $G$ generated by $y$, $g$, and $V$ is nonabelian with noncyclic center, so must be equal to $G$. The same is true if we replace $g$ by $g^i\neq 1$, so by minimality $g$ has prime order $p$. The group $H$ generated by $y$ and $g$ is isomorphic to $C_p\rtimes C_{2^m}$. Since some power of $y$ lies in $V$, the group $G$, which is generated by $H$ and $V$, must be isomorphic to $(C_p\rtimes C_{2^m} )\times C_2$. If $m=1$ this group has representation dimension 2, so $m>1$. A faithful 3-dimensional representation can be constructed similarly to the previous case.  
\end{proof}

Remark. The proof of part b) broke into the cases $|G^1|=2$ and $|G^1|\geq4$, and showed that $|G^1|$ is always a power of 2. The final classification reveals that in fact we must have $|G^1|=2$ in the case $r=1$.

Next we shall show that if $G$ is a minimally faithful 2-group of degree 3 then $Z(G)$ is not cyclic. This will be useful in the subsequent result that classifies the minimally faithful 2-groups of degree 3.

\begin{lemma}
Suppose $G$ is a nonabelian $2$-group and minimally faithful of degree 3. Then $Z(G)$ is not cyclic. 
\end{lemma}

\begin{proof}
Suppose $Z(G)$ is cyclic and $\pi:G\to \GL_2(\C)\times \GL_1(\C)$ is a faithful 3-dimensional representation, noting that no 3-dimensional representation of $G$ is irreducible since $G$ is a 2-group. Define the two subgroups $$N = \ker(G\xra{\pi} \GL_2(\C)\times \GL_1(\C) \to \GL_1(\C) )$$ $$K=\ker(G\xra{\pi} \GL_2(\C)\times \GL_1(\C) \to \GL_2(\C) ).$$ Both $K$ and $N$ are nontrivial since $\rdim(G)=3$, while $N\cap K=1$. Clearly $K\subset Z(G)$ since $\pi(K)$ consists of matrices of the form $\diag(1, 1, c)\in Z(\GL_2(\C)\times \GL_1(\C))$. Let $z\in Z(G)$  be the unique element of order 2. For any subgroup $M\subset G$ we have $z\in M$ iff $M\cap Z(G)\neq 1$. Since $K\subset Z(G)$, $z\in K$, so $z\not\in N$ and $N\cap Z(G)=1$. Since $N$ is a 2-group it has nontrivial center. The subgroup $Z(G)N \cong Z(G)\times N$ therefore has noncyclic center, and must therefore be abelian (it cannot equal $G$ since $Z(G)$ is cyclic). So $N$ is abelian, and in fact $N$ must be cyclic: othweise $N\supset C_2^2$ and $Z(G)N\supset C_2^3$. Let $n\in N$ be the unique element in $N$ of order 2. Since $N$ is a normal subgroup of $G$, $n\in Z(G)$, $n\neq z$, a contradiction.   
\end{proof}







\begin{proposition}
Let $G$ be a nonabelian 2-group. 

a) If $G$ is minimally faithful of degree 3 and contains a nonabelian proper subgroup then $G = Q_8\times C_2$. 

b) Suppose $G$ contains only abelian proper subgroups. Then $G$ is minimally faithful of degree 3 if and only if every index-2 subgroup of $G$ has two invariant factors. 
\end{proposition}

\begin{proof}

a) Let $N$ be a proper nonabelian subgroup of $G$, and $Z=Z(G)$, which we have shown is not cyclic. Since $NZ$ is a nonabelian group with noncyclic center, $\rdim(NZ)\geq 3$ and $G = NZ$. This remains true if we replace $Z$ by any noncyclic subgroup of $Z$. So there exists a subgroup $V\cong C_2\times C_2\subset Z$ with $G = NV$. If $N\cap V$ was trivial then any element in $N$ of order 2 would, with $V$, generate $C_2^3$. Thus $N\cap V$ has order 2, since $V\not\subset N$. This implies $G \cong N\times C_2$. If $N$ contained a nonabelian proper subgroup $N_1$, then $N_1\times C_2$ would be nonabelian with noncyclic center, hence would have representation dimension 3 and contradicting minimality of $G$. Since $\rdim(N)=2$ and $N\not\supset C_2\times C_2$, lemma \ref{2-group rdim2 lemma} forces $N=Q_8$.

b) Let $H\subset G$ be a subgroup of index 2. If $H$ is cyclic then there is a faithful character $\chi:H\to \C^\times$, and then $\Ind_H^G(\chi)$ is a faithful 2-dimensional representation of $G$, which is imposisble. The subgroup $H$ cannot have more than 2 invariant factors since $G\not\supset C_2^3$. The first direction is proved.

Now assume that $G$ is a nonabelian 2-group whose index-2 subgroups are all abelian with two invariant factors. Let $H\subset G$ of index 2, so $H= C_{2^r}\times C_{2^s} = \langle x, y\rangle $ where $r\geq  s>0$. Let $a\in G\setminus H$. Since proper subgroups of $G$ are abelian, $a$ must centralize every proper subgroup of $H$ that it normalizes. So every characteristic subgroup of $H$ is contained in $Z(G)$, and thus $C_{2^{r-1}}\times C_{2^s}$ is central 
if $r>s$ and $C_{2^{s-1}}\times C_{2^{s-1}}$ is central if $r=s$. So $x^2, y^2\in Z(G)$ if $r=s$ and $x^2, y\in Z(G)$ if $r>s$.

If $a^2= 1$ then $\langle a, x^{2^{r-1}}, y^{2^{s-1}}\rangle=C_2^3$. This is a proper subgroup of the nonabelian group $G$, contradicting our assumption that all proper subgroups have at most 2 invariant factors. So $a^2\neq 1$. Now suppose $r=s=1$. Since $1\neq a^2\in H = C_2\times C_2$, we have $a^4=1$ and $|G|=[G:H]|H| = 8$. The subgroup generated by $a$ has index 2 in $G$, contradicting the assumption that all index-2 subgroups have two invariant factors. Thus either $r\geq s>1$ or $r>s=1$. In either case $Z(G)\cap H$ contains all elements of $H$ order 2, hence $Z(G)$ is not cyclic and $\rdim(G)\geq 3$. All that remains is to construct a faithful 3-dimensional representation of $G$, since our assumption implies $\rdim(H)=2$ for all maximal subgroups $H$ of $G$. 

We may assume without loss of generality that $x\not\in Z(G)$; this is automatic if $r>s$. We have $axa^{-1}=x^i y^j$ for some $i, j\geq 0$. Since $x^2 = ax^2a^{-1} = (x^iy^j)^2 = x^{2i}y^{2j}$, we have $j \in \{0, 2^{s-1} \}$ and $i\in \{1,  1+ 2^{r-1} \}$. In any case, $z:=axa^{-1}x^{-1}$ is of order 2, so $z\in Z(G)$. Since $x$ and $a$ do not commute, they generate $G$. Therefore $G/\langle z\rangle$ is abelian, and $z$ generates $[G, G]$.

Let $\chi:H\to \C^\times$ with $\chi(z)= -1$. Then $\rho:=\Ind_H^G(\chi)$ is a 2-dimensional representation of $G$. It is clear that $\ker(\rho)\cap aH$ is empty, so $\ker(\rho)\subset H$. If $\ker(\rho)$ was not cyclic then it would contain the three elements in $H$ of order 2, but $z\not\in \ker(\rho)$. We have $\ker(\rho)\neq 1$ since $\rdim(G)\geq 3$. Let $k$ be the unique element of order 2 in $\ker(\rho)$ and $\nu:G\to \C^\times$ with $\nu(k)=-1$; such a $\nu$ exists since the image of $k$ in $G/[G, G] = G/\langle z\rangle$ is nontrivial. The representation $\rho\oplus \nu$ is faithful. 
\end{proof}

The $2$-groups of the form described 
here have been classified - see \cite{MM03} and \cite{M01}. For example, the only group of order at most 16 on this list is the nontrivial semidirect product $C_4\rtimes C_4$. Since the actual construction of these groups would involve a large detour, we content ourselves to provide the relevant reference. 


All that remains is to classify the odd-order groups that are minimally faithful of degree 3. This naturally breaks into the cases of 3-groups and not 3-groups. 

\begin{proposition}
Suppose $G$ is a nonabelian $3$-group and minimally faithful of degree 3. Then either 

a) $G$ is the Heisenberg group of order 27, or  

b) $G = \langle a, b|\  a^{3^k} = b^3 =1 , \  b^{-1}ab = a^{1+3^{k-1}}  \rangle = C_{3^k}\rtimes C_3$ for some $k\geq 2$. 
\end{proposition}

\begin{proof}
Let $H$ be an index-3 subgroup, so $H$ is abelian and normal. If every index-3 subgroup of $G$ is cyclic then every proper subgroup is cyclic, which would force $G$ to be cyclic since it has odd order. So we may assume $H$ is not cyclic. Let $B\subset H$ be the subgroup generated by elements of order $3$, so $B\cong C_3\times C_3$. Since $Z(G)$ is cyclic, $B\not\subset Z(G)$. Let $a\in G\setminus C_G(B)$. Since $a$ and $B$ generate a nonabelian group, they generate $G$. Since $B$ is characteristic in $H$, $a$ normalizes $B$. 
We can choose generators $z$ and $b$ of $B$ so that $az = za$ and $aba^{-1} = bz$. In particular, $G$ is generated by $a$ and $b$. Suppose $a$ has order $3^k$. Since $a^3, z\in Z(G)$ and $Z(G)$ is cyclic, either $k=1$ or else $k\geq 2$ and $z = a^{3^{k-1}}$. The former case yields the Heisenberg group, which is easily shown to be minimally faithful of degree 3. The latter case yields the groups $C_{3^k}\rtimes C_3$ with the claimed presentations. Every subgroup is clearly abelian of rank at most 2. Let $\chi:\langle a \rangle \to \C^\times$ be a faithful character. Then $\Ind_{\langle a \rangle }^G (\chi)$ is a faithful 3-dimensional representation of $G$. 
\end{proof}

\begin{proposition}
Suppose $G$ is a nonabelian group of odd order, is not a 3-group, and is minimally faithful of degree 3. Then there exists a prime $p\geq 5$ such that either 

a) $G = C_p \rtimes C_{3^n}$ if $p\equiv 1 \pmod 3$, where $C_{3^{n-1}}$ is the center, or 

b) $G = (C_p\times C_p)\rtimes C_{3^n}$ if $p\equiv -1 \pmod 3$, where $C_{3^{n-1}}$ is the center. 
\end{proposition}

\begin{proof}
We know that $3$ divides $|G|$ and every proper subgroup of $G$ is abelian. From proposition \ref{AbelianSubgroupsReferenceProposition}, we have $|G| = 3^a p^b$ for some prime $p>3$. Let $P$ be a $p$-Sylow subgroup and $Q$ a $3$-Sylow subgroup. We know one of them is cyclic and the other normal and elementary abelian (of rank at most 2). Since $|\Aut(C_3^2)| =2^4\cdot3$ and $|\Aut(C_3)|=2$, the subgroup $Q$ cannot be normalized by $P$ without being centralized by $P$. Thus $Q$ is cyclic and $P$ is elementary abelian. A generator $x$ of $Q$ must act on $P$ by an automorphism of order 3, in order that $G$ contains no proper nonabelian subgroups. 

Suppose $P= C_p$. Then $3$ divides $|\Aut(P)|=p-1$ and $G=C_p\rtimes C_{3^n}$. One obtains $\rdim(G) = 3$ by inducing a faithful character of the index-3 cyclic subgroup $C_p\cdot Z(G) \cong C_{3^{n-1}p}$. 

Suppose $P= C_p\times C_p$. Then $3$ divides $|\Aut(P)| =(p+1)p(p-1)^2$ and $G=(C_p\times C_p)\rtimes C_{3^n}$. If $p\equiv 1\pmod 3$, let $\alpha, \alpha^2\in \Z/p\Z$ be the two roots of $t^2 + t+1=0\pmod p$. Then $P = P_0\times P_1\times P_2$ where $P_i:=\{y\in P: xyx^{-1} = y^{\alpha^i} \}$. If there exists $1\neq y\in P_1$ or $1\neq y\in P_2$ then the subgroup generated by $x$ and $y$ will necessarily be a nonabelian proper subgroup of $G$. But if $P=P_0$ then $G$ is abelian. Thus $p\equiv -1 \pmod 3$. We need only show that $G$ has a faithful 3-dimensional representation. Let $H =C_p\times C_p\times C_{3^{n-1}} $ and $\chi:H\to \C^\times$ have kernel of order $p$. Let $\rho:=\Ind_H^G(\chi)$. Then $\dim(\rho)=[G:H] =3$. It is easy to see that $\ker(\rho)\subset H$, and from there that $\ker(\rho)\subset \ker(\chi)$. If $1\neq y\in \ker(\rho)$ then $xyx^{-1}\not\in \ker(\chi)$ since $xyx^{-1}$ does not lie in the subgroup generated by $y$. So $xyx^{-1}\not\in \ker(\rho)$, which contradicts the fact that $\ker(\rho)$ is a normal subgroup of $G$. Thus $\rho$ is faithful even though $\chi$ is not. 
\end{proof}

\section{Further Work}

There are several natural questions that one can pose about the class of minimally faithful finite groups $G$ of degree $n$, for $n>3$. Are they always solvable? Which $p$-groups are minimally faithful? Which $G$ can be realized inside $\SL_n(\C)$? It is perhaps computationally intractable to obtain a full classification for substantially larger $n$.

In another direction, the notion of minimally faithful groups of degree $n$ can be readily defined over an arbitrary field $F$ instead of $\C$. There is still much that can be said for minimally faithful groups of degree 2; we state the results here without proof. Any finite subgroup of $\GL_1(F)$ is cyclic. It follows from this and simple calculations in $\GL_2(\overline{F})$ that the only candidates for nonabelian minimally faithful groups of degree 2 are a subset of those for $\C$, namely $Q_8$ and $C_p\rtimes C_{2^n}$; these work if and only if $F$ contain a 4th root of unity and a $2^np$th root of unity, respectively. In particular, if $F$ has characteristic 2 then all minimally faithful groups of degree 2 over $F$ are abelian. An abelian minimally faithful group over $F$ need not be elementary abelian. For example, one can have $G=C_{p^a}$ with $F$ of characteristic $p$ and containing a $p^{a-1}$-root of unity but not a $p^a$th root of unity, and such that the polynomial $(t^{p^a}-1)/(t^{p^{a-1}}-1)$ has an irreducible quadratic factor. A root $r$ of such a quadratic factor generates a quadratic extension $F(r)/F$, and multiplication by $r$, as an element of $\Aut_F(F(r))\cong \GL_2(F)$, generates a faithful representation of $C_{p^a}$. This arises, for example, if $F=\Q$ and $p^a=4$. Similarly, if $F$ has characteristic $p$ then $C_p$ is not a subgroup of $\GL_1(F)$ but is a subgroup of $\GL_2(F)$ via the matrices $\begin{bmatrix} 1 & i \\ 0 & 1 \end{bmatrix}$, $0\leq i\leq p-1$. Finally, one has the groups $C_p\times C_p$, which will be minimally faithful of degree 2 if and only if $F$ contains a $p$th root of unity. 

One could also broaden the class of groups and representations considered beyond the finite class. If $G$ is an arbitrary group, one can define its representation dimension to be the minimum cardinality $\kappa$ such that $G$ admits a faithful representation on a vector space of dimension $\kappa$. Then one could define $G$ to be minimally faithful of degree $\kappa$ if, for all proper subgroups $H$ of $G$, the representation dimension of $H$ is a strictly smaller cardinal than $\kappa$. For example, if $F$ is a finite field and $G=\Z_{p^\infty}$ is the Pr{\"u}fer $p$-group (the group of $p$-power torsion in $\C^\times$) then $G$ has infinite representation dimension over $F$, while every proper subgroup $H$ of $G$ is finite (and cyclic), so $H$ has finite representation dimension over $F$. Thus $G$ is minimally faithful over any finite field.

\end{document}